\documentclass[10pt,a4paper]{amsart}

\usepackage{amsmath}
\usepackage{amsfonts}
\usepackage{amssymb}
\usepackage{mathrsfs}
\usepackage{amsthm}
\theoremstyle{plain}
\newcommand{\C}{\mathbb{C}}
\newcommand{\F}{\mathscr{F}}

\newcommand{\sing}{\textsf{Sing}}

\newcommand{\ord}{\textsf{ord}}

\newcommand{\bb}{\text{BB}}
\newcommand{\gsv}{\text{GSV}}
\newcommand{\supp}{\text{Supp}}
\newcommand{\cs}{\text{CS}}
\newcommand{\im}{\text{Im}}
\newcommand{\var}{\text{Var}}
\newcommand{\tr}{\text{Tr}}
\newcommand{\res}{\text{Res}}

\newcommand{\mc}[1]{\mathcal{#1}}

\newcommand{\R}{\mathbb{R}}
\newtheorem{theorem}{Theorem}[section]

\newtheorem{lemma}[theorem]{Lemma}
\newtheorem{proposition}[theorem]{Proposition}
\newtheorem{corollary}[theorem]{Corollary}
\theoremstyle{definition}
\newtheorem{definition}{Definition}[section]
\newtheorem{example}{Example}[section]
\newtheorem{remark}{Remark}[section]

\title[Lehmann-Suwa Residues of holomorphic foliations]{Lehmann-Suwa Residues of codimension one holomorphic foliations and applications}
\author{Arturo Fern\'andez-P\'erez}
\address[Arturo Fern\'andez-P\'erez]{Departamento de Matem\'atica, Universidade Federal de Minas Gerais, UFMG}
\curraddr{Av. Ant\^onio Carlos 6627, 31270-901, Belo Horizonte-MG, Brazil.}
\email{fernandez@ufmg.br}
\author{Jimmy T\'amara}
\address[Jimmy T\'amara]{Instituto de Matem\'atica y Ciencias Afines IMCA}
\curraddr{Jr. Los Bi\'ologos 245, Lima, Per\'u}
\email{jimmy.tamara@imca.edu.pe}



\begin{document}

\begin{abstract}
Let $\F$ be a singular codimension one holomorphic foliation on a compact complex manifold $X$ of dimension at least three such that its singular set has codimension at least two. In this paper, we determine \textit{Lehmann-Suwa residues} of $\F$ as multiples of complex numbers by integration currents along irreducible complex subvarieties of $X$. We then prove a formula that determines the Baum-Bott residue of \textit{simple almost Liouvillian foliations of codimension one}, in terms of Lehmann-Suwa residues, generalizing a result of Marco Brunella. As an application, we give sufficient conditions for the existence of dicritical singularities of a singular real-analytic Levi-flat hypersurface $M\subset X$ tangent to $\F$.
\end{abstract}

\maketitle
\section{Introduction}
\par In 1999, D. Lehmann and T. Suwa \cite{suwa} gave a generalization to the case of arbitrary dimension, of the \textit{variational index}, defined by Khanedani and Suwa \cite{variation} for singular holomorphic foliations on complex surfaces. More precisely, Lehmann and Suwa proved the following result. 
\begin{theorem}[Lehmann-Suwa \cite{suwa}]\label{Lehmann-Suwa_theo}
Let $V$ be a complex subvariety of dimension $m\geq 1$ in a complex manifold $X$ and let $\F$ be a singular holomorphic foliation of dimension $k\geq 1$ on $X$ which leaves $V$ invariant. Denote by $\mathcal{N}_{\F}$ the normal sheaf of $\F$. Let $\varphi$ be a homogeneous symmetric polynomial of degree $d>m-k$. 
\begin{enumerate}
\item For each compact connected component $Z$ of the singular set $\sing(\F|_V)$, there exists a homology class
$$\res_{\varphi}(\F,\mathcal{N}_{\F}|_V;Z)\in H_{2m-2d}(Z;\mathbb{C}),$$
which is determined by the local behavior of $\F$ near $Z$. 
\item If $V$ is compact, 
$$\sum_{Z}(i_{Z})_{*}\res_{\varphi}(\F,\mathcal{N}_{\F}|_{V};Z)=\varphi(\mathcal{N}_{\F})\smallfrown [V]\,\,\,\,\,\,\,\,\text{in}\,\,\,\,\,\,\,H_{2m-2d}(V;\mathbb{C}),$$
where $i_{Z}$ denotes the embedding $Z\hookrightarrow V$ and the sum is taken over all the components $Z$ of $\sing(\F|_V)$. 
\end{enumerate} 
\end{theorem}
\par When $\varphi=c_1$, the expression $\res_{\varphi}(\F,\mathcal{N}_{\F}|_{V};Z)$ is called the \textit{variation of $\F$ with respect to $V$ at $Z$}. In general, the computation and determination of these residues is a difficult problem and few results are known. For example, if the foliation $\F$ is singular at $p\in\C^2$ and $V$ is a reduced complex curve through  $p$ invariant by $\F$. Then the variation of $\F$ relative to $V$ at $p$ is given by
$$\res_{c_1}(\F,\mathcal{N}_{\F}|_V;p)=\var(\F,V,p) [p],$$ where $\var(\F,V,p)$ is the \textit{variational index} 
defined by Khanedani and Suwa in \cite{variation}. When $\F$ is a one-dimensional holomorphic foliation on a complex manifold, that is, locally defined by holomorphic vector fields, there is an explicit formula in terms of Grothendieck residues for $\res_{\varphi}(\F,\mathcal{N}_{\F}|_{V};p)$, see for instance \cite{suwa} and \cite{suwa_livro}.   
\par This paper aims to study of residues of codimension one holomorphic foliations on complex manifolds of dimension at least three. First, we will restrict our attention to \textit{Lehmann-Suwa residues} (or \textit{variations}) of a codimension one holomorphic foliation $\F$ on a compact complex manifold $X$ of dimension at least three. In Section \ref{variational_section}, it is shown that Lehmann-Suwa residues localized at codimension two irreducible components of the singular set of 
$\F$ can be determined as multiples of complex numbers by integration currents along of these irreducible components. 
\par In \cite{cs_lins}, Lins Neto introduced the Camacho-Sad index \cite{CS} for a codimension one holomorphic foliation concerning a \textit{codimension one complex submanifold}, and Gmira \cite{gmira} obtained a generalization of some results due to  Lins Neto \cite{cs_lins}. Recently, Corr\^ea and Machado \cite{machado} defined the \textit{GSV-index} for holomorphic Pfaff systems on complex manifolds generalizing the \textit{GSV-index} of G\'omez-Mont--Seade--Verjovsky \cite{GSV}. In Section \ref{gsv}, 
combining the Corr\^ea-Machado index with the Lehmann-Suwa residues, we recover the Camacho-Sad index for a codimension one holomorphic foliation $\F$ with respect to a codimension one complex subvariety $V$ (possibly with singularities).  
\par In \cite{perrone}, Brunella and Perrone determine the Baum-Bott residue \cite{baum} of a codimension one holomorphic foliation concerning a singular component of codimension two via integration over a 3-sphere of a certain  3-form (see for instance Section \ref{section_baum}). In general, 
the determination of Baum-Bott residues (in terms of the Grothendieck residues) of singular holomorphic foliations of arbitrary codimension have been obtained by Corr\^ea and Louren\c{c}o \cite{correa}. In Section \ref{liouvillian}, we will prove (see Theorem \ref{almost_foliation}) that the Baum-Bott and Lehmann-Suwa residues are related when the codimension one foliation $\F$ is a \textit{simple almost Liouvillian foliation} (see Definition \ref{def_almost}). 
\par In the last part of the paper, we apply our residual formulas to prove, under certain conditions, the existence of \textit{dicritical singularities} of a real-analytic Levi-flat hypersurface tangent to a codimension one holomorphic foliation on a compact complex manifold of complex dimension at least three.
\par It is important pointing out that a general construction of residue theorems for holomorphic foliations and Pfaff systems of any dimension can be found in \cite{abate}, \cite{suwa_livro} and \cite{sancho}. In the special case of singular codimension one holomorphic distributions, we refer the reader to \cite{izawa} and the references given there.\\

\noindent {\bf Acknowledgments.} 
The authors wish to express his gratitude to Maur\'icio Corr\^ea and Miguel Rodr\'iguez Pe\~na for several helpful comments concerning to work. The authors also gratefully acknowledge the anonymous referee for giving many suggestions that helped to improve the presentation of the paper. The first author was partially supported by CNPq Brazil grant number 427388/2016-3 and PRONEX/FAPERJ and the second author was partially supported by Fondecyt-Per\'u CG 217-2014.  
\section{Holomorphic foliations}
\par Let $X$ be a complex manifold and $TX$ the holomorphic tangent bundle of $M$. Let 
$\Theta_X=\mathcal{O}(TM)$ be the sheaf of holomorphic vector fields on $X$. A \textit{singular holomorphic foliation} $\F$ of dimension $r$ on $X$ is determined by a coherent subsheaf $\Theta_{\F}\subset\Theta_X$ of rank $r$, which is involutive (or integrable), i.e., such that 
$$\left[\Theta_{\F,p},\Theta_{\F,p}\right]\subset \Theta_{\F,p}\,\,\,\,\,\,\,\,\,\text{for all}\,\,\,p\in X.$$
We set $\mathcal{N}_{\F}=\Theta_{X}/\Theta_{\F}$ and define $S(\F)$ by 
$$S(\F)=\sing(\mathcal{N}_{\F}).$$ Note that $S(\F)$ is an analytic subset and, from \cite{suwa}, $S(\F)$ is describe as follows: let $U$ be a sufficiently small coordinate neighborhood with coordinates $(z_1,\ldots,z_n)$ and let $v_1,\ldots,v_r$ be generators of $\F$ on $U$. We write $v_i=\displaystyle\sum^{n}_{j=1}f_{ij}(z)\frac{\partial}{\partial{z_{j}}}$. Then 
$$S(\F)\cap U=\{z\in U: rank\left(f_{ij}(z)\right)<r\}.$$
Furthermore, the foliation $\F$ induce an exact sequence 
$$0\longrightarrow\Theta_{\F}\longrightarrow\Theta_{X}\longrightarrow\mathcal{N}_{\F}\longrightarrow 0.$$
\par In this paper, we study foliations of \textit{codimension one} in $X$, i.e., foliations of dimension $\dim(X)-1$. As is common, codimension one foliations can be described dually utilizing differential 1-forms: a \textit{codimension one singular holomorphic foliation $\F$} on $X$
is determined by a saturated locally free subsheaf
$$N^{*}_{\F}\subset\Omega^{1}_X$$ of rank one,
which satisfies the \textit{Frobenius integrability condition}. Locally, $N^{*}_{\F}$ is generated by holomorphic 1-forms $\omega_{k}\in\Omega^1_X(U_{k})$, where $\{U_{k}\}_{k\in I}$ is an open covering of $X$, such that $$\omega_{k}\wedge d\omega_{k}=0$$
and $$\omega_k=g_{k\ell}\omega_{\ell}\,\,\,\,\,\,\,\,\,\,\text{on}\,\,\,U_{k}\cap U_{\ell}.$$
The functions $g_{k\ell}$ are nowhere vanishing, and the multiplicative cocycle $\{g_{k\ell}\}$ defines a line bundle $N_{\F}$, called  the \textit{normal bundle} of $\F$. The \textit{singular set} $\sing(\F)$ of $\F$ is the analytic subset of $X$ defined by 
$$\sing(\F)\cap U_{k}=\text{zeros of}\,\,\, \omega_k,\,\,\,\,\,\,\,\,\,\,\,\,\,\forall\,\,k\in I.$$
The saturated condition means that the zero set of every $\omega_k$ has codimension at least two. Therefore, by definition, $\sing(\F)$ has codimension at least two. We will denote $\sing_2(\F)$ the union of all irreducible components of $\sing(\F)$ of codimension two. 
\par Throughout this paper, we will always work with $N_\F$ and $\sing(\F)$. The relation between $S(\F)$ and $\sing(\F)$ can be found in \cite{Junya} and \cite{suwa_livro} as well. Moreover, we will  assume by hypotheses that $S(\F)=\sing(\F)$ and $\sing_2(\F)\neq \emptyset$. 
\par We remark that in general the characteristic classes of $\mathcal{N}_\F$ (as in Theorem \ref{Lehmann-Suwa_theo}) and $N_{\F}$ are not the same. However, $c_1(\mathcal{N}_{\F})$ and $c_1(N_\F)$ are equal in the $K$-group $K(X)$. In fact, we have the exact sequence 
$$0\longrightarrow N^{*}_{\F}\longrightarrow\Omega^1_{X}\longrightarrow\Omega^1_X/ N^{*}_{\F}\longrightarrow 0.$$ Taking the duals in the last sequence, we obtain an exact sequence
$$0\longrightarrow(\Omega^1_{X}/N^{*}_{\F})^{*}\longrightarrow\Theta_{X}\longrightarrow N_{\F}\longrightarrow \mathcal{E}xt^1(\Omega^1_{X}/N^{*}_{\F},\mathcal{O}_X)\longrightarrow 0.$$
We have $(\Omega^1_{X}/N^{*}_{\F})^{*}=\Theta_\F$ and an exact sequence
$$0\longrightarrow \mathcal{N}_\F\longrightarrow N_{\F}\longrightarrow \mathcal{E}xt^1(\Omega^1_{X}/N^{*}_{\F},\mathcal{O}_X)\longrightarrow 0.$$
The characteristic classes of a coherent sheaf are defined by taking a resolution of the sheaf by vector bundles and regarding it as an element in the $K$-group $K(X)$. Since $c_1$ is additive on $K(X)$, we have
$$c_1(\mathcal{N}_\F)=c_1(\Theta_X)-c_1(\Theta_\F)=-c_1(\Omega^1_X)+c_1(\Omega^1_X)-c_1(N^{*}_\F)=c_1(N_\F).$$
\section{Lehmann-Suwa formula}\label{variational_section}
 Let $\F$ be a codimension one singular holomorphic foliation on a compact complex manifold $X$ of dimension at least three and let $V\subset X$ be a complex hypervariety invariant by $\F$. Here, \textit{complex hypervariety} means codimension one complex subvariety and \textit{invariant} means that if a point of $V$ belongs to the regular part of $\F$, then the whole leaf through this point is included in $V$.  We shall assume furthermore that $V$ is \textit{reduced}, that is, the divisor $V$ does not contain multiple irreducible components.
\par Let us denote by $\sing(V)$ the singular set of $V$ and set $$\sing_2(\F,V)=\sing(V)\cup(\sing_2(\F)\cap V).$$ Let $Z$ be an irreducible component of $\sing_2(\F,V)$ such that $Z$ has pure codimension two. Take a generic point $p\in Z$, that is, a point where $Z$ is smooth and disjoint from the other singular components. 
Take $B_p\subset X$ a small ball centered at $p$ such that $Z\cap B_p$ is the unique irreducible component of $\sing_2(\F,V)\cap B_p$ and suppose that $\omega\in\Omega^{1}_X(B_p)$ represents $\F$ in $B_p$. Working with smooth sections of $N^{*}_{\F}$, instead of holomorphic ones, the corresponding cohomology group is trivial, and so we can certainly  find a smooth (1,0)-form $\beta\in A^{1,0}(B^{*}_p)$ such that 
\begin{equation}\label{godbillon}
d\omega=\beta\wedge \omega,
\end{equation}
 where  $B^{*}_p=B_p\setminus (\sing_2(\F,V)\cap B_p)$. Since $p$ is a generic point of $Z$ and the codimension of $Z$ with respect to $V$ is one, we may take a one-dimensional small transverse section $\sum_{p}$ to $Z$ at $p$ such that $\sum_p\subset V$. Then we define 
\begin{equation}\label{index_vari}
\var(\F,V,Z):=\frac{1}{2\pi i}\int_{\Gamma}\beta,
\end{equation}
 where $\Gamma$ is a generator of $H_1(\sum_{p}\setminus\{p\},\mathbb{Z})$. We call this complex number the \textit{Variational index of $\F$ concerning $V$ along $Z$}. By a connectedness argument, it does not depend on the choice of the generic point $p\in Z$. It is the natural extension of the variational index of Khanedani-Suwa \cite{variation}. 
\par The following result is a particular case of Theorem \ref{Lehmann-Suwa_theo}, the novelty will be to obtain a proof using the Variational index given in (\ref{index_vari}).
\begin{theorem}\label{suma_variational}
Let $\F$ be a codimension one holomorphic foliation on a compact complex manifold $X$ of dimension at least three and let $V\subset X$ be a reduced complex hypervariety invariant by $\F$ such that $\sing_2(\F,V)\neq\emptyset$. 
Then
$$\sum_{Z}\var(\F,V,Z)[Z]=c_1(N_{\F}|_V)\smallfrown [V],$$
where the sum is done over all irreducible components $Z$ of $\sing_2(\F,V)$ and $[Z]$ denotes the integration current associated to $Z$.
\end{theorem}
\begin{proof}
We cover $X$ by open subsets $U_{k}$ where the foliation $\F$ is defined by integrable holomorphic 1-forms $\omega_{k}$, with $\omega_{k}=g_{k\ell}\omega_{\ell}$, where $g_{k\ell}\in\mathcal{O}^{*}(U_{k}\cap U_{\ell})$ whenever $U_{k}\cap U_{\ell}\neq\emptyset$. 
Assume that $V\cap U_{k}=\{f_{k}=0\}$, where $f_k\in\mathcal{O}(U_k)$.  On $U_{k}\cap U_{\ell}$, we have $f_{k}=\varphi_{k\ell}f_{\ell}$ with $\varphi_{k\ell}\in\mathcal{O}^{*}(U_{k}\cap U_{\ell})$ and the cocycle $\{\varphi_{k\ell}\}$ defines the line bundle $[V]$ on $X$. 
\par We may find smooth $(1,0)$-forms $\gamma_{k}$ on $U^{*}_{k}=U_{k}\setminus(\sing_2(\F,V)\cap U_{k})$ such that $d\omega_{k}=\gamma_{k}\wedge\omega_{k}$. We fix a small neighborhood $U$ of $\sing_2(\F,V)$ and we regularize each $\gamma_{k}$ on $U$, that is, we choose a smooth $(1,0)$-form $\tilde{\gamma}_{k}$ on $U_{k}$ coinciding with $\gamma_{k}$ outside of $U_{k}\cap U$. Then the smooth $(1,0)$-forms  $$\zeta_{k\ell}=\frac{dg_{k\ell}}{g_{k\ell}}-\tilde{\gamma}_{k}+\tilde{\gamma}_{\ell}$$ vanish on $\F$ outside $U$. This cocycle can be trivialized 
$$\zeta_{k\ell}=\zeta_{k}-\zeta_{\ell},$$ where $\zeta_{k}$ is a smooth $(1,0)$-form on $U_{k}$ vanishing on $\F$ outside of $U_{k}\cap U$. Hence, after setting $\hat{\gamma}_{k}=\tilde{\gamma}_{k}+\zeta_{k}$, we get
\begin{equation}
\frac{dg_{k\ell}}{g_{k\ell}}=\hat{\gamma}_{k}-\hat{\gamma}_{\ell}.
\end{equation}
Note that we still have $d\omega_{k}=\hat{\gamma}_{k}\wedge\omega_{k}$ outside of $U_{k}\cap U$. The globally defined closed 2-form (of mixed type $(2,0)+(1,1)$) 
$$\Omega=\frac{1}{2\pi i} d\hat{\gamma}_{k}$$
represents, in the De Rham cohomology, the first Chern class of $N_\F$. Moreover, outside $U$, $\Omega$ vanishes when restricted to leaves of $\F$ ($\Omega\wedge\omega_k=0$), and in particular, when restricted to $V$, except on small neighborhoods of $\sing_2(\F,V)$ in $V$. This means that $$\supp(\Omega|_{V})\subset \overline{U}.$$
\par Let $\psi$ be a closed smooth $(2n-4)$-form on $V$ and let $\langle\,,\,\rangle$ be a hermitian metric on $[V]$. Let $\sigma$ be the global section of $[V]$ defined by $\sigma|_{V_{k}}=f_{k}$, where $V_{k}:=V\cap U_{k}=\{f_k=0\}$. Set $S=\sing_2(\F,V)$. We consider the tubular neighborhood of $S$ in $V$ for some small number $\epsilon>0$ as follows
$$T_S(\epsilon)=\{p\in V:||\sigma(p) ||_{p}\leq \epsilon\}.$$
Note that $\sigma(p)=0$ if, and only if, $p\in S$, moreover $$\partial{T}_{S}(\epsilon)=\{p\in V: ||\sigma(p)||_p=\epsilon\}.$$ 
Assume that $S=\displaystyle\bigcup^{m}_{j=1}Z_j$. For each $Z_j$ choose a small neighborhood $V_j$ in $U$ such that $Z_j\subset V_j$, $\overline{V_j}\subset U$ and $\supp(\Omega|_V)=\displaystyle\bigcup^{m}_{j=1}V_j$.  Denoting $T_{Z_{j}}(\epsilon)=T_{S}(\epsilon)\cap \overline{V_j}$,  we have $Z_j\subset T_{Z_{j}}(\epsilon)$. Therefore, 
\begin{eqnarray*}
\int_V \Omega\wedge \psi& = & \sum^m_{j=1}\int_{V_j}\Omega\wedge\psi\\
& = & \frac{1}{2\pi i}\sum^{m}_{j=1}\int_{V_{j}}d\hat{\gamma}_{j}\wedge\psi \\
& = &  \frac{1}{2\pi i}\sum^{m}_{j=1}\left[\int_{T_{Z_{j}}(\epsilon)}d\hat{\gamma}_{j}\wedge\psi+\int_{V_j-T_{Z_{j}}(\epsilon)}d\hat{\gamma}_{j}\wedge\psi\right].
\end{eqnarray*}
Since $\displaystyle\lim_{\epsilon\to 0}\int_{V_j-T_{Z_{j}}(\epsilon)}d\hat{\gamma}_{j}\wedge\psi=0$,
we get
\begin{eqnarray}\label{integral1}
\int_V \Omega\wedge \psi& = & \frac{1}{2\pi i}\sum^{m}_{j=1}\lim_{\epsilon\to 0}\int_{T_{Z_{j}}(\epsilon)}d\hat{\gamma}_{j}\wedge\psi \nonumber\\
& = &  \frac{1}{2\pi i}\sum^{m}_{j=1}\lim_{\epsilon\to 0}\int_{T_{Z_{j}}(\epsilon)}d(\hat{\gamma}_{j}\wedge\psi)\nonumber\\
& = & \frac{1}{2\pi i}\sum^{m}_{j=1}\lim_{\epsilon\to 0}\int_{\partial{T}_{Z_{j}}(\epsilon)}\hat{\gamma}_{j}\wedge\psi.
\end{eqnarray}
Now, take a smooth point $p\in Z_{j}-\displaystyle\bigcup_{\ell\neq j}Z_{\ell}$, then there exists a neighborhood $W_p\subset V_j$ of $p$ and a coordinate system $(z_1,z_2,\ldots,z_{n-1})$ centered at $p$ such that $Z_j\cap W_p=\{z_1=0\}$
and $\partial{T}_{Z_j}(\epsilon)\cap W_p=\{|z_1|=\epsilon, z'\in\Delta_{\epsilon}\}$, where $z'=(z_2,\ldots,z_{n-1})$ and $\Delta_{\epsilon}=\{z'\in\mathbb{C}^{n-2}:|z'|\leq \epsilon\}$. Note that
\begin{eqnarray*}
\lim_{\epsilon\to 0} \frac{1}{2\pi i}\int_{\partial{T}_{Z_{j}}(\epsilon)\cap W_p}\hat{\gamma}_{j}\wedge\psi
& = & \lim_{\epsilon\to 0} \frac{1}{2\pi i}\int_{\Delta_{\epsilon}}\int_{|z_1|=\epsilon}\hat{\gamma}_{j}\wedge\psi  \\
& = & \lim_{\epsilon\to 0} \int_{\Delta_{\epsilon}}\left [\frac{1}{2\pi i}\int_{|z_1|=\epsilon}\hat{\gamma}_{j}\right ]\psi\\
& = & \lim_{\epsilon\to 0} \int _{\Delta_{\epsilon}}\var(\F,V,Z_j)\psi\\
& = & \var(\F,V,Z_j)\lim_{\epsilon\to 0} \int _{\Delta_{\epsilon}}\psi\\
&=&  \var(\F,V,Z_j)\int _{Z_j\cap W_p}\psi.
\end{eqnarray*}
Therefore, 
\begin{eqnarray}\label{integral2}
\lim_{\epsilon\to 0} \frac{1}{2\pi i}\int_{\partial{T}_{Z_{j}}(\epsilon)}\hat{\gamma}_{j}\wedge\psi & = & \var(\F,V,Z_j)\int _{Z_j}\psi\nonumber\\
& = & \var(\F,V,Z_j)[Z_j](\psi).
\end{eqnarray}
Hence, from (\ref{integral1}) and (\ref{integral2}) we get
$$\int_{V}\Omega\wedge\psi=\sum^{m}_{j=1}\var(\F,V,Z_j)[Z_j](\psi),$$
for any closed smooth $(2n-4)$-form $\psi$ on $V$.
Using Poincar\'e duality and the fact that $\Omega|_{V}$ represents, in the De Rham cohomology, the Chern class of $N_{\F}|_V$, we obtain
\begin{equation*}
\sum^{m}_{j=1}\var(\F,V,Z_j)[Z_j]=c_1(N_{\F}|_V)\smallfrown [V].
 \end{equation*}
 \end{proof}
\begin{remark}
The proof above gives more, namely, the Lehmann-Suwa residues are determined as follows:
$$\res_{c_1}(\F,N_{\F}|_V;Z)=\var(\F,V,Z)[Z].$$
\end{remark}

\section{GSV and Camacho-Sad indices for codimension one holomorphic foliations}\label{gsv}
\subsection{Saito's decomposition} 
The following lemma can be found in Saito \cite[Section 1]{saito}. When $\F$ is a germ of holomorphic foliation at $0\in\C^2$, we refer to reader to \cite{lins}, \cite{suwa_livro}.
\begin{lemma}[Saito \cite{saito}]\label{saito_lemma}
Let $\F$ be a germ of codimension one singular holomorphic foliation at $0\in\C^n$, $n\geq 2$, defined by a germ of an integrable holomorphic 1-form $\omega$. Suppose $V=\{f=0\}$ is a germ at $0\in\C^n$ of reduced complex hypervariety invariant by $\F$. Then
 there exist germs of holomorphic functions $g$, $h$ and a germ of holomorphic 1-form $\eta$ at $0\in\C^n$ such that 
\begin{equation}\label{eq_saito}
g\omega=hdf+f\eta,
\end{equation}
where $h$ and $f$ have no common factors. Moreover, $g$ and $f$ also have no common factors. 
\end{lemma} 
\par Using the Saito's decomposition, we can now state a similar result to Brunella \cite[Proposition 5]{index}.
\begin{proposition}\label{prop-variational}
Let $\F$, $V$ and $Z$ be as in Lemma \ref{saito_lemma}. Then 
$$\var(\F,V,Z)=\frac{1}{2\pi i}\int_{\Gamma}\left(\frac{g}{h}d\left(\frac{h}{g}\right)-\frac{\eta}{h}\right),$$
where $\Gamma$ is as in equation (\ref{index_vari}). 
\end{proposition}
\begin{proof}
By Lemma \ref{saito_lemma} we have 
\begin{eqnarray*}
\omega&=&\frac{h}{g}df+f\frac{\eta}{g}.
\end{eqnarray*}
Therefore, $$d\omega=d\left(\frac{h}{g}\right)\wedge df+df\wedge\frac{\eta}{g}+fd\left(\frac{\eta}{g}\right).$$ Restringing to $V$, we get 
\begin{eqnarray}
d\omega&=&d\left(\frac{h}{g}\right)\wedge df+df\wedge\frac{\eta}{g}\label{eq1}
\end{eqnarray}
and 
\begin{eqnarray}
\omega&=&\frac{h}{g}df.\label{eq2}
\end{eqnarray}
From (\ref{eq1}) and (\ref{eq2}) it follows that 
\begin{eqnarray*}
d\omega&=&\left(d\left(\frac{h}{g}\right)-\frac{\eta}{g}\right)\wedge df\\
&=&\left(\frac{g}{h}d\left(\frac{h}{g}\right)-\frac{\eta}{h}\right)\wedge\omega.
\end{eqnarray*}
Hence $$\var(\F,V,Z)=\frac{1}{2\pi i}\int_{\Gamma}\left(\frac{g}{h}d\left(\frac{h}{g}\right)-\frac{\eta}{h}\right),$$ where $\Gamma$ is a curve as in (\ref{index_vari}).
\end{proof}

\subsection{GSV-index for codimension one holomorphic foliations}
A. G. Aleksandrov in \cite{aleksandrov} introduced the concept of \textit{multiple residues} of logarithmic differentials forms and generalize the Saito's decomposition theorem \cite{saito}. Using Aleksandrov's decomposition theorem, Corr\^ea and Machado defined in \cite{machado} the GSV-index for holomorphic Pfaff systems. In this subsection, we particularizing this definition for codimension one holomorphic foliations.

\par Let $\F$ be a germ of codimension one singular holomorphic foliation at $0\in\C^n$, $n\geq 3$, defined by a germ of an integrable holomorphic 1-form $\omega$. Suppose $V=\{f=0\}$ is a germ at $0\in\C^n$ of reduced complex hypervariety invariant by $\F$. Then by Lemma \ref{saito_lemma} we have $$g\omega=hdf+f\eta.$$
For each irreducible component $Z$ of $\sing_2(\F,V)$, Corr\^ea and Machado \cite{machado} defined the \textit{GSV-index} as follows:
\begin{equation}
\gsv(\F,V,Z):=\ord_{Z}\left(\frac{h}{g}\Big{|}_{V}\right).
\end{equation}
According to Corr\^ea-Machado \cite[Theorem 3.2]{machado} we can formulate:
\begin{theorem}[Corr\^ea-Machado \cite{machado}]\label{gsv-index}
Let $\F$ be a codimension one holomorphic foliation on a compact complex manifold $X$ of dimension at least three and let $V\subset X$ be a reduced complex hypervariety invariant by $\F$ such that $\sing_2(\F,V)\neq\emptyset$. Denote by $N_{V/X}$ the normal bundle of $V$ in $X$. 
Then
\begin{equation}
\sum_{Z}\gsv(\F,V,Z)[Z]=c_1(N_{\F}|_V\otimes (N_{V/X})^{-1})\smallfrown [V],
\end{equation}
where the sum is done over all irreducible components $Z$ of $\sing_2(\F,V)$ and $[Z]$ denotes the integration current associated to $Z$.
\end{theorem}
\subsection{Camacho-Sad index for codimension one holomorphic foliations}  Define the \textit{Camacho-Sad index} as follows: 
\begin{equation}\label{cs-index}
\cs(\F,V,Z):=\var(\F,V,Z)-\gsv(\F,V,Z).
\end{equation}
When $V$ is smooth, $CS(\F,V,Z)$ coincide with the index defined by Lins Neto \cite{cs_lins}, see also \cite{gmira}. Note that Theorem \ref{suma_variational} and Theorem \ref{gsv-index} implies the following result of global nature. 
\begin{theorem}
Let $\F$ be a codimension one holomorphic foliation on a compact complex manifold $X$ of dimension at least three and let $V\subset X$ be a reduced complex hypervariety invariant by $\F$ such that $\sing_2(\F,V)\neq\emptyset$. 
Then
$$\sum_{Z}\cs(\F,V,Z)[Z]=c_1(N_{V/X})\smallfrown [V]$$
where the sum is done over all irreducible components $Z$ of $\sing_2(\F,V)$ and $[Z]$ denotes the integration current associated to $Z$.
\end{theorem}

\section{Baum-Bott index}\label{section_baum}
In this section, we define the Baum-Bott index, following \cite{perrone}. Similarly to above section, we work  with smooth sections of $N^{*}_{\F}$, instead of holomorphic ones, then there exists a smooth (1,0)-form $\beta\in A^{1,0}(B^{*}_p)$ such that 
$d\omega=\beta\wedge \omega$, where $\omega$ is a local generator of $N^{*}_{\F}$. The smooth 3-form (of mixed type $(3,0)+(2,1)$) 
\begin{equation}\label{bb}
\frac{1}{(2\pi i)^2}\beta\wedge d\beta
\end{equation}
is closed, and it has a De Rham cohomology class in $H^{3}(B^{*}_p,\mathbb{C})$, which does not depend on the choice of $\omega$ and $\beta$.
\par Let $Z$ be an irreducible component of $\sing_2(\F)$. Take a generic point $p\in Z$ and pick $B_p$ sufficiently small ball, so that $S(B_p):=\sing_2(\F)\cap B_p$ is a codimension two subball of $B_p$. Then the above De Rham class can be integrated over an oriented 3-sphere $L_p\subset B^{*}_p$ positively linked with $S(B_p)$: 
$$\bb(\F,Z)=\frac{1}{(2\pi i)^2}\int_{L_{p}}\beta\wedge d\beta.$$
This complex number is called \textit{Baum-Bott residue of $\F$ along $Z$}. Again by a connectedness argument, it does not depend on the choice of the generic point $p\in Z$. 
\par Let us recall that every irreducible component $Z$ of $\sing_2(\F)$ has a class $[Z]\in H^4(X,\mathbb{C})$ (conveniently defined via the integration current over $Z$). Therefore, we have the following result.
\begin{theorem}[Baum-Bott \cite{baum}, Brunella-Perrone \cite{perrone}]\label{baum-bott}
$$\sum_{Z}\bb(\F,Z)[Z]=c_1^{2}(N_\F)$$
where the sum is done over all irreducible components of $\sing_2(\F)$.
\end{theorem}

\section{Almost Liouvillian foliations}\label{liouvillian}
In this section, we consider germs at $0\in\C^n$, $n\geq 3$, of singular holomorphic foliations of codimension one. Let $\F$ be a germ at $0\in\C^n$ of a codimension one holomorphic foliation such that $0\in\sing(\F)$ and suppose that $\F$ is defined by a germ of an integrable holomorphic 1-form $\omega\in\Omega^{1}(\C^n,0)$. 
Let $\sing_2(\F)$ be the germ at $0\in\C^n$ defined by the union of the germs at $0\in\C^n$ of irreducible components of the singular set of $\F$ whose codimension is precisely two. In this section, we assume that $\sing_2(\F)$ is not empty. 
\begin{definition}\label{def_almost}
We say that the germ $\F$ is an \textit{almost Liouvillian foliation} at $0\in\C^n$ if there exists a germ of closed meromorphic 1-form $\gamma_0$ and a germ of holomorphic 1-form $\gamma_1$ at $0\in\C^n$ such that 
\begin{equation}\label{poles}
d\omega=(\gamma_0+\gamma_1)\wedge \omega.
\end{equation}
We say that $\F$ is a \textit{simple almost Liouvillian foliation} at $0\in\mathbb{C}^n$ if we can choose $\gamma_0$ having only first-order poles.
\end{definition}
The next lemma was proved by Brunella \cite{index} in the two-dimensional case. We extend this fact for high dimension. 
\begin{lemma}
If $\F$ is almost Liouvillian foliation at $0\in\C^n$ defined by $\omega\in\Omega^{1}(\C^n,0)$, $n\geq 3$. Then the poles divisor of $\gamma=\gamma_0+\gamma_1$ is invariant by $\F$.
\end{lemma}
\begin{proof}
Let $V=(\gamma)_{\infty}$ be the poles divisor of $\gamma$. If $p$ is a smooth point of $V$ such that $p\not\in\sing(\F)$, then there exists a coordinate system $(z_1,\ldots,z_n)$ at $p$ such that $z_1(p)=\ldots=z_n(p)=0$  and $\omega=a(z_1,\ldots,z_n)dz_n$, where $a\in\mathcal{O}^{*}(\C^n,0)$. Let $\gamma=b_1dz_1+\ldots+b_ndz_n$, where $b_1,\ldots,b_n$ are germs of meromorphic functions at $0\in\C^n$. It follows from (\ref{poles}) that $b_1,\ldots,b_{n-1}\in\mathcal{O}(\C^n,0)$ and therefore $\gamma_0=b_ndz_n$. Since $\gamma_0$ is closed, we get $b_n=b_n(z_n)$ and $\gamma_0$ may be written as $$\gamma_0=\frac{h(z_n)}{z^{k}_n}dz_n,$$ where $h(z_n)$ is a holomorphic function and $k\geq 1$. Consequently, the germ of $V$ at $p$ is given by $\{z_n=0\}$, which implies that $V$ is invariant by $\omega$.
\end{proof}
\begin{remark}\label{var_zero}
Let $\F$ be a germ at $0\in\C^n$, $n\geq 3$, of a codimension one holomorphic foliation such that $Z$ is a germ at $0\in\C^n$ of an irreducible component of $\sing_2(\F)$. Suppose that there exists a germ at $0\in\C^n$ of a complex hypervariety $V$ invariant by $\F$ such that $V$ does not contain $Z$. Then it is not difficult to see that the definition of index variational $\var(\F,V,Z)$ (see for instance (\ref{index_vari})) may be extended to an irreducible component $Z$ of $\sing_2(\F)$ that is not contained in $V$. In this case, we have $\var(\F,V,Z)=0$.
\end{remark}
\par The next theorem extends a result due to Brunella \cite[Proposition 8]{index}. This result provides an effective way of computing Baum-Bott residues of codimension one holomorphic foliations in high dimension. We remark that 
germs at $0\in\C^3$ of codimension one holomorphic foliations with reduced singularities (see for instance Cano \cite{cano}), logarithmic foliations and some transversally affine foliations are examples of simple almost Liouvillian foliations. 

\begin{theorem}\label{almost_foliation}
Let $\F$ be a germ at $0\in\mathbb{C}^n$, $n\geq 3$, of a simple almost Liouvillian foliation defined by $\omega\in\Omega^1(\C^n,0)$ such that $$d\omega=(\gamma_0+\gamma_1)\wedge\omega.$$
 Let $V$ be the divisor of poles of $\gamma=\gamma_0+\gamma_1$ and $V_1,\ldots,V_{\ell}$ the irreducible components of $V$.         
Let $Z$ be an irreducible component of $\sing_2(\F)$. Then 
$$\bb(\F,Z)=\sum^{k}_{j=1}\res(\gamma_0,V_j) \var(\F,V_j,Z),$$
where $V_1,\ldots,V_k$ are the irreducible components of $V$ that contains $Z$.
\end{theorem}
\begin{proof}
Take a generic point $p\in Z$ and pick $B_p$ a sufficiently small ball such that $S(B_p)$ is a codimension two subball of $B_p$ at $Z$ (as in Section \ref{section_baum}). Let $S^3_{\epsilon}\subset B^{*}_p$ be an oriented 3-sphere positively linked with $S(B_p)$.   
Let $\partial{V}_j:=S^3_{\epsilon}\cap V_j$ and  let $W_j$ be a tubular neighborhood of $\partial{V}_j$ such that $W_j\cap Z=\emptyset$. Then there exists holomorphic 1-form $\gamma'_j$ in a neighborhood of $W_j$ such that 
\begin{equation}\label{index_1}
d\omega=\gamma'_j\wedge \omega.
\end{equation}
Note that $W:=\displaystyle\bigcup^{N}_{j=1}W_j$ is a tubular neighborhood of $\partial{V}=S^{3}_{\epsilon}\cap V$ and there is an partition of unity $\rho=\{\rho_j\}$ for $W$ subordinate to open cover $\{W_j\}$. With this we can define 
$$\gamma'=\rho_1\gamma'_1+\ldots+\rho_N\gamma'_N.$$ It is easily seen that $\gamma'$ is a holomorphic 1-form in a neighborhood of $W$.
\par Let $\phi\in C^{\infty}_{c}(W)$ be equal to 1 on a smaller neighborhood of $\partial{V}$. Then $\beta=\phi\gamma'+(1-\phi)\gamma'$ is a smooth $(1,0)$-form on a neighborhood of $S^3_{\epsilon}$. 
Note that $$d\omega=\beta\wedge \omega.$$
Denoting  $\beta_j=\beta|_{S^3_{\epsilon}\cap W_j}$ and $\phi_j=\phi|_{S^3_{\epsilon}\cap W_j}$ for each $j=1,\ldots, N$, we get 
$$\beta_j=\phi_j\gamma'_j+(1-\phi_j)\gamma'_j,\,\,\,\,\,\,\,\,\,\text{on}\,\,\,\,\,\, S^3_{\epsilon}\cap W_j.$$
Note also that $S^3_{\epsilon}\cap W=\displaystyle\bigcup^{N}_{j=1}(S^3_{\epsilon}\cap W_j)$.
To continue, we may choose holomorphic coordinates $(z_1,\ldots,z_n)$ near each $\partial{V}_j=S^3_{\epsilon}\cap V_j$, with $z_1$ varying on a neighborhood of the unitary circle and $(z_2,\ldots,z_n)$ on a neighborhood of the origin of $\mathbb{C}^{n-1}$ such that 
\begin{eqnarray*}
V_j&=&\{z_n=0\},\\ 
\partial{V}_j&=&\{|z_1|=1,z_2=z_3=\ldots=z_n=0\},\\
S^3_{\epsilon}\cap W_j&=&\{|z_1|=1,|z_n|\leq\epsilon,z_2=z_3=\ldots=z_{n-1}=0\},\\ 
\partial(S^3_{\epsilon}\cap W_j)&=&\{|z_1|=1,|z_n|=\epsilon,z_2=z_3=\ldots=z_{n-1}=0\}.
\end{eqnarray*}
We claim $Supp(\beta\wedge d\beta)\subset S^3_{\epsilon}\cap W$. In fact, by construction we have 
\begin{equation}
\beta \wedge d\beta=\beta_j\wedge d\beta_j\,\,\,\,\,\,\,\text{in}\,\,\,\,\,\,S^3_{\epsilon}\cap W_j.
\end{equation}
On the other hand, since $$\beta_j\wedge d\beta_j=\phi^2\gamma'_j\wedge d \gamma'_j-\gamma'_j\wedge d\phi\wedge\gamma_j+\phi(1-\phi)\gamma'_j\wedge d\gamma_j+(1-\phi)\phi\gamma_j\wedge d\gamma'_j+(1-\phi)^2\gamma_j\wedge d\gamma_j,$$ and $\gamma'_j\wedge d\gamma'_j=\gamma_j\wedge d\gamma'_j=\gamma'_j\wedge d\gamma_j=\gamma_j\wedge d\gamma_j=0$ in $S^3_{\epsilon}\cap W_j$, we get 
$$\beta \wedge d\beta=\beta_j\wedge d\beta_j=d\phi\wedge\gamma'_j\wedge\gamma_j.$$ Therefore, $Supp(\beta\wedge d\beta)\subset S^3_{\epsilon}\cap W$ and the assertion is proved. 
\par Now 
\begin{equation}\label{index_3}
\int_{S^3_{\epsilon}}\beta\wedge d\beta=\int_{S^3_{\epsilon}\cap W}\beta\wedge d\beta=\sum^{N}_{j=1}\int_{S^3_{\epsilon}\cap W_j}\beta_j\wedge d\beta_j.
\end{equation}
Since $\beta_j\wedge d\beta_j=d\phi\wedge\gamma'_j\wedge\gamma_j$ in $S^3_{\epsilon}\cap W_j$, we obtain 
$d((1-\phi)\gamma\wedge\gamma'_j)=\beta_j\wedge d\beta_j$. Then 
\begin{eqnarray}\label{index_4}
\int_{S^3_{\epsilon}\cap W_j}\beta_j\wedge d\beta_j&=&\int_{S^3_{\epsilon}\cap W_j}d((1-\phi)\gamma\wedge\gamma'_j)\nonumber\\
&=& \int_{\partial{(S^3_{\epsilon}\cap W_j)}}(1-\phi)\gamma\wedge\gamma'_j\nonumber\\
&=&\int_{\partial{(S^3_{\epsilon}\cap W_j)}}\gamma\wedge\gamma'_{j}\nonumber\\
&=&\int_{\partial{(S^3_{\epsilon}\cap W_j)}}(\gamma_0+\gamma_1)\wedge\gamma'_{j}\nonumber\\
&=&\int_{\partial{(S^3_{\epsilon}\cap W_j)}}\gamma_0\wedge\gamma'_j.
\end{eqnarray}
In the coordinate system $(z_1,\ldots,z_n)$, we have 
$$\gamma=b_1dz_1+\ldots+b_{n-1}dz_{n-1}+b_n dz_{n}$$ where $\gamma_{1}=b_1dz_1+\ldots+b_{n-1}dz_{n-1}$ and $\gamma_0=b_ndz_n$. Furthermore, since $\F$  is a simple almost Liouvillian foliation, we have $$\gamma_0=\lambda_j\frac{dz_n}{z_n}+\gamma_{0j},$$ where $\lambda_j=\res(\gamma_0,V_j)$ and $\gamma_{0j}$ is a suitable holomorphic 1-form. 
On the other hand, 
$\gamma'_j=a_{1}dz_1+\ldots+a_ndz_n,$ with $a_j\in\mathcal{O}(W_j)$ for all $i=1,\ldots,n,$ and in particular 
$\gamma'_{j}|_{\partial{V_j}} =a_1(z_1,0,\ldots,0)dz_1$, where $\partial{V}_j=\{|z_1|=1,z_2=\ldots=z_n=0\}$.
Then 
\begin{eqnarray}\label{index_2}
\int_{\partial(S^3_{\epsilon}\cap W_{j})}\gamma_0\wedge\gamma'_{j}&=&\int_{|z_1|=1,|z_n|=\epsilon}\left(\lambda_j\frac{dz_n}{z_n}+\gamma_{0j}\right)\wedge(a_1dz_1+a_ndz_n)\nonumber\\
&=&\int_{|z_1|=1,|z_n|=\epsilon}\left(\lambda_ja_1\frac{dz_n}{z_n}\wedge dz_1\right)\nonumber\\
&=&(2\pi i)\lambda_j\int_{|z_1|=1}\left(\frac{1}{2\pi i}\int_{|z_n|=\epsilon}\frac{a_1(z_1,0,\ldots,z_n)}{z_n}dz_n\right)dz_1\nonumber\\
&=&(2\pi i)\lambda_j\int_{|z_1|=1}a_1(z_1,0,\ldots,0)dz_1\nonumber\\
&=&(2\pi i)\lambda_j\int_{\partial{V_j}}\gamma'_j|_{\partial{V_j}}.
\end{eqnarray}
If $V_j$ contains $Z$ then it follows from (\ref{index_1}) that 
$$\frac{1}{2\pi i}\int_{\partial{V_j}}\gamma'_j|_{\partial{V_j}}= \var(\F,V_j,Z).$$ Thus from (\ref{index_2}) we get 
$$\frac{1}{(2\pi i)^2}\int_{\partial(S^3_{\epsilon}\cap W_{j})}\gamma_0\wedge\gamma'_{j}= \lambda_j\var(\F,V_j,Z).$$
If $V_j$ does not contain $Z$ then $\var(\F,V_j,Z)=0$ by Remark \ref{var_zero}. 
Finally, from (\ref{index_3}) and (\ref{index_4}), we obtain
$$\bb(\F,Z)=\sum^{k}_{j=1}\res(\gamma_0,V_j) \var(\F,V_j,Z),$$
where $V_1,\ldots,V_k$ are the irreducible components of $V$ that contains $Z$.
\end{proof}
\par To end this section we give an example where Theorem \ref{almost_foliation} applies.
\begin{example}
Let $\F$ be the germ at $0\in\C^3$ of a holomorphic foliation defined by $$\omega=2yzdx+3xzdy+4xydz.$$
We have 
$$d\omega=\gamma_0\wedge\omega,\quad\text{where}\quad\gamma_0=-\frac{dx}{x}-2\frac{dy}{y}-3\frac{dz}{z}.$$
In particular, $\F$ is a codimension one Liouvillian foliation at $0\in\C^3$. Let $V=(\gamma_0)_{\infty}=\displaystyle\bigcup^3_{j=1} V_j$, where $V_1=\{x=0\}$, $V_2=\{y=0\}$ and $V_3=\{z=0\}$. Note that 
$$\res(\gamma_0,V_1)=-1,\,\,\,\,\,\,\,\,
\res(\gamma_0,V_2)=-2,\,\,\,\,\,\,\,\,\,
\res(\gamma_0,V_3)=-3.$$
Let $Z=\{y=z=0\}$, it is evident that $Z\subset\sing_2(\F)$. Furthermore $Z\subset V_2$ and $Z\subset V_3$. To compute $\var(\F,V_3,Z)$, we pick $p=(1,0,0)\in Z$ and the transverse section $\sum_p=\{|y|\leq 1,x=1,z=0\}$ to $Z$ in $V_3$. By Proposition \ref{prop-variational}, we get 
$$\var(\F,V_3,Z)=\frac{1}{2\pi i}\int_{\Gamma_3}\left(\frac{d(4xy)}{4xy}-\frac{2ydx+3xdy}{4xy}\right)|_{x=1}=\frac{1}{4},$$ where $\Gamma_3\in H_1(\sum_p\setminus\{p\},\mathbb{Z})$. 
\par On the other hand, to compute $\var(\F,V_2,Z)$, take again $p=(1,0,0)\in Z$ and the transverse section $\sum_p=\{|z|\leq 1,x=1,y=0\}$ to $Z$ in $V_2$. Again by Proposition \ref{prop-variational}, we get 
$$\var(\F,V_2,Z)=\frac{1}{2\pi i}\int_{\Gamma_2}\left(\frac{d(3xz)}{3xz}-\frac{2zdx+4xdz}{3xz}\right)|_{x=1}=-\frac{1}{3},$$ where $\Gamma_2\in H_1(\sum_p\setminus\{p\},\mathbb{Z})$.
Hence, applying Theorem \ref{almost_foliation}, we conclude
$$\bb(\F,Z)=(-2)\left(-\frac{1}{3}\right)+(-3)\left(\frac{1}{4}\right)=-\frac{1}{12}.$$
\par Using a recently result of Corr\^ea-Louren\c{c}o \cite{correa}, we can verify that the above computations are correct. In fact, as in \cite[Example 4.1]{correa}, let us consider $p=(1,0,0)\in Z$, $D=\{|(y,z)|\leq 1, x=1\}$ and 
$$\omega|_{D}=3zdy+4ydz.$$
The dual vector field of $\omega|_{D}$ is $X=4y\frac{\partial}{\partial{y}}-3z\frac{\partial}{\partial{z}}$. A straightforward calculation shows that 
$$JX(0,0)=\begin{bmatrix}
    4       & 0 \\
    0   &  -3 
   \end{bmatrix}.$$
Thus $$\bb(\F,Z)=\frac{\tr(JX(0,0))^2}{\det (JX(0,0))}=-\frac{1}{12}.$$
\end{example}

\section{Singular holomorphic foliations tangent to singular Levi-flat hypersurfaces}\label{Levi-flat}
\par Motived by \cite{andres} and \cite{projective}, we study singular codimension one holomorphic foliations tangent to singular real-analytic Levi-flat hypersurfaces in compact complex manifolds with emphasis on the type of singularities of them. 
\par Let us clarify these terms. A  closed  set $M$   of a
  complex manifold   $X$ is a \emph{real-analytic subvariety} if it is defined, in some   neighborhood of each point of $M$, by the vanishing of finitely many real-analytic functions with real values. We say that a real-analytic subvariety $M$ is \emph{irreducible} if it cannot be written as the union of two real-analytic subvarieties properly contained in it. If $M$ is irreducible, it has a well-defined dimension $\dim_{\mathbb{R}} M$. A \emph{hypervariety} is a subvariety of real codimension one.
\par If $M \subset X$ is a real-analytic submanifold of real codimension one.
For each   $p \in M$, there is a unique complex hyperplane $L_{p}$ contained in the tangent space $T_{p}M \subset T_{p}X$. This  defines a real-analytic distribution $p \mapsto L_{p}$ of complex hyperplanes in $T M$.  When this distribution is integrable in the sense of Frobenius, we say that $M$ is a {\em Levi-flat hypersurface}. In this case, $M$ is foliated by   immersed complex manifolds of dimension $n-1$. This foliation, denoted by $\mathcal{L}$, is known as  {\em Levi foliation}. A normal form for  such an object was given by
E. Cartan \cite[Theorem IV]{cartan}: at each $p \in M$, there are   holomorphic coordinates $(z_{1},\ldots,z_{n})$ in a neighborhood $U$ of $p$
such that
\begin{equation}\label{formalocal-hlf}
M \cap U = \{\im(z_{n}) = 0\}.
\end{equation}
 As a consequence, the leaves of
$\mathcal{L}$
 have local equations $z_{n} = c$, for $c \in \mathbb{R}$.
\par In the singular case,    an irreducible real-analytic   hypervariety $M \subset X$ is said to be \emph{Levi-flat} if its \emph{regular part}  is a Levi-flat hypersurface. We denote by
$M_{reg}$ its \emph{regular part} --- the points near which $M$ is a real-analytic manifold of dimension equal to
$\dim_{\R}M$. Let $\sing(M)$ be the singular points of $M$, points near which $M$ is not a real-analytic submanifold (of any dimension). Because we are working with real-analytic sets, the set $\sing(M)$ is not in general equal to the complement of $M_{reg}$ as defined above, and is only a semi-analytic set (see for instance \cite{singularlebl}).
If $M \subset X$ is a real-analytic Levi-flat hypervariety,
Cartan's local trivialization   allows the extension of the Levi foliation  to a non-singular holomorphic foliation  in a neighborhood of  $M_{reg}$ in $X$, which is unique as a germ around $M_{reg}$.
In general, it is not possible to extend    $\mathcal{L}$   to a singular holomorphic foliation in a neighborhood of $M$. There are examples of Levi-flat hypervarieties whose Levi foliations extend to singular   $k$-webs in the ambient space \cite{generic}.  However,
there is an extension  in some ``holomorphic lifting'' of $M$ (see for instance \cite{brunella}).
If   a singular holomorphic foliation $\F$ in the ambient space $X$ coincides with the     Levi foliation on $M_{reg}$, we say either that $M$ is \emph{invariant} by $\F$ or that $\F$ is \emph{tangent} to $M$.
\begin{definition}
 A singular point $p\in \sing(M)$ is called \textit{dicritical} if, for every neighborhood $U$ of $p$, infinitely many leaves of the Levi-foliation on $M^{*}\cap U$ have $p$ in their closure. 
 \end{definition}
 
\par Recently dicritical singularities of singular real-analytic Levi-flat hypersurfaces have been characterized in terms of the \textit{Segre varieties}, see for instance Pinchuk-Shafikov-Sukhov \cite{pinchuk}. 
\par We recall the definition of meromorphic and holomorphic first integral for holomorphic foliations. Let $\F$ be a singular holomorphic foliation on $X$. Recall that $\F$ admits a \textit{meromorphic} (\textit{holomorphic}) first integral at $p\in X$, if there exists a neighborhood $U$ of $p$ and  
a \textit{meromorphic} (\textit{holomorphic}) function $h$ defined in $U$ such that its indeterminacy (zeros)
set is contained in $\sing(\F)\cap U$ and its level curves contain the leaves of $\F$ in $U$.

\par To prove the main result of this section, we need the following result. 
\begin{theorem}[Cerveau-Lins Neto \cite{alcides}]\label{lins-cerveau}
 Let $\F$ be a germ of codimension one holomorphic foliation at $0\in\mathbb{C}^{n}$, $n\geq{2}$, tangent to a germ of an irreducible real-analytic hypersurface $M$. Then $\F$ has a non-constant meromorphic first integral. In the case of dimension two, we can precise more:
\begin{enumerate}
\item If $\F$ is dicritical then it has a non-constant meromorphic first integral.
\item If $\F$ is non-dicritical then it has a non-constant holomorphic first integral.
\end{enumerate}
\end{theorem}
\par We now prove a generalization of \cite[Lemma 3.2]{andres}.
To prove this we use Theorem \ref{almost_foliation} and Theorem \ref{lins-cerveau}. 
\begin{proposition}\label{baum-bott_signo}
Let $\F$ be a germ of a codimension one holomorphic foliation at $0\in\mathbb{C}^n$, $n\geq 3$. Suppose that $\sing_2(\F)\neq\emptyset$ and that $\F$ has a non-constant holomorphic first integral, then for every irreducible component $Z$ of $\sing_2(\F)$, we have 
$$\bb(\F,Z)\leq 0.$$
\begin{proof}
Let $g=g_1^{m_1}g_2^{m_2}\ldots g_k^{m_{k}}$ be a germ at $0\in\C^n$ of a holomorphic first integral for $\F$, where $g_1,\ldots,g_k$ are irreducible germs at $0\in\C^n$ and $m_1,\ldots,m_k$ are non-negative integers. Then the germ 
$\omega=m_1g_2\ldots g_{k}dg_1+\ldots+m_kg_1\ldots g_{k-1}dg_k$ at $0\in\C^n$ defines $\F$. Since $dg=h\omega$ with $h=g_1^{m_1-1}g_2^{m_2-1}\ldots g_{k}^{m_k-1}$ we get
\begin{eqnarray}
d\omega=-\frac{dh}{h}\wedge \omega,
\end{eqnarray}
where $\frac{dh}{h}=(m_1-1)\frac{dg_1}{g_1}+\ldots+(m_k-1)\frac{dg_k}{g_k}$. In particular, $\F$ is a simple Liouvillian foliation at $0\in\C^n$. 
\par Let $Z$ be an irreducible component of $\sing_2(\F)$, and let $V_j=\{g_j=0\}$. Note that $Z\subset V_j$ for all $1\leq j\leq k$. To compute $\bb(\F,Z)$, we need to compute $\var(\F,V_j,Z)$. By Proposition \ref{prop-variational} we get
\begin{eqnarray*} 
\var(\F,V_j,Z)&=&\frac{1}{2\pi i}\int_{\partial{V}_j}\left[\sum^{k}_{\ell\neq j}\frac{dg_{\ell}}{g_{\ell}}-\sum^{k}_{\ell\neq j}\frac{m_{\ell}}{m_{j}}\frac{dg_{\ell}}{g_{\ell}}\right]\\
&=&\frac{1}{2\pi i}\int_{\partial{V}_j}\sum^{k}_{\ell\neq j}\left(1-\frac{m_{\ell}}{m_j}\right)\frac{dg_{\ell}}{g_{\ell}}\\
&=&\sum^{k}_{\ell\neq j}\left(1-\frac{m_{\ell}}{m_j}\right)\ord_Z(g_{\ell}|_{V_j}).
\end{eqnarray*}
On the other hand, $\res\left(-\frac{dh}{h},V_j\right)=-(m_j-1)=1-m_j$. According to Theorem \ref{almost_foliation}, we obtain 
\begin{eqnarray*}
\bb(\F,Z)&=&\sum^{k}_{j=1}\res\left(-\frac{dh}{h},V_j\right)\var(\F,V_j,Z)\\
&=& \sum^{k}_{j=1}(1-m_j)\left[\sum^{k}_{\ell\neq j}\left(1-\frac{m_{\ell}}{m_j}\right)\ord_{Z}(g_{\ell}|_{V_j})\right].
\end{eqnarray*}
Since $\ord_{Z}(g_{\ell}|_{V_{j}})=\ord_Z(g_{j}|_{V_{\ell}})$ for $\ell\neq j$, we get  
\begin{eqnarray}
\bb(\F,Z)=-\sum_{1\leq \ell< j\leq k}\frac{(m_{\ell}-m_j)^2}{m_{\ell}m_{j}}\ord_Z(g_{j}|_{V_{\ell}})\leq 0.
\end{eqnarray}
\end{proof}
\end{proposition}
\par Now we state the main result of this section.
\begin{theorem}\label{Levi}
Let $\F$ be a codimension one singular holomorphic foliation on a compact complex manifold $X$ of dimension at least three, tangent to an irreducible real-analytic Levi-flat hypervariety $M\subset X$. Suppose that:
\begin{enumerate}
\item $Sing_2(\F)$ is not empty and $\sing_2(\F)\subset M$,
\item $h^{4}(X,\mathbb{C})=1$ and denote by $\zeta$ the generator of $H^4(X,\C)$,
\item for every fundamental class $[W]\in H^4(X,\C)$ of an irreducible complex subvariety $W\subset X$ of codimension two, there exists $\alpha>0$ such that $[W]=\alpha\zeta$,
\item for the Chern class $c_1^2(N_\F)\in H^4(X,\C)$, there exists $\alpha_0>0$ such that $$c_1^2(N_\F)=\alpha_0\zeta.$$
\end{enumerate}
Then there exists an irreducible component $Z$ of $\sing_2(\F)$ such that it contains some dicritical point $p\in\sing(M)$. Moreover, $\F$ has a non-constant meromorphic first integral at $p$.
\end{theorem}
\begin{proof}

Suppose by contradiction that $\sing_2(\F)$ consists of irreducible components with only non-dicritical singularities of $M$. Take an irreducible component  $Z$ of $\sing_2(\F)$ and a generic point $q\in Z$. By hypothesis (1) we have $Z\subset M$.
Let $U$ be a small neighborhood of $q$ in $X$  such that $\F$ is represented by a holomorphic 1-form $\omega$ on $U$ and $Z\cap U$ is the unique singular component of $\omega$. 
 Then, since $\F$ and $M$ are tangent in $U$ we have $\F|_{U}$ admits a meromorphic first integral $g$ on $U$, by Theorem \ref{lins-cerveau}. But since $q\in U$ is a non-dicritical singularity, $g$ must be a holomorphic first integral.
 \par Applying Proposition \ref{baum-bott_signo} to $\F|_U$, we get $\bb(\F,Z)\leq 0$, for any $Z\subset\sing_2(\F)$. Assume that $\sing_2(\F)=\displaystyle\bigcup^{k}_{j=1}Z_j$. Then Baum-Bott's formula (cf. Theorem \ref{baum-bott}) implies that
\begin{eqnarray*}
c^{2}_{1}(N_{\mc{F}})&=&\sum^{k}_{j=1}\bb(\F,Z_j)[Z_j],\quad\text{in}\quad H^{4}(X,\C)\\
&=&\left(\sum^{k}_{j=1}\bb(\F,Z_j)\alpha_j\right)\zeta,\quad\text{for some}\quad\alpha_j>0
\end{eqnarray*}
which is absurd with hypothesis $(4)$, because $\alpha_0=\displaystyle\sum^{k}_{j=1}\bb(\F,Z_j)\alpha_j\leq 0$. Therefore, there exists an irreducible component $Z$ of $\sing_2(\F)$ such that it contains some dicritical point $p\in M$. Applying again Theorem \ref{lins-cerveau}, we obtain a non-constant meromorphic first integral for $\F$ in a neighborhood of $p$.
\end{proof}
\par When $X=\mathbb{P}^n$, the complex projective space, $n\geq 3$, we recall the singular set of any codimension one holomorphic foliation on $\mathbb{P}^n$, $n\geq 3$, has an irreducible component of codimenion two, see for instance \cite[Proposition 2.6, page 95]{jouanolou}.
\par  Let $h$ be the hyperplane class in $\mathbb{P}^n$. Then 
$H^4(\mathbb{P}^n,\mathbb{C})$ is generated by $h^2$. Thus for every codimension two irreducible component $W$ in $\mathbb{P}^n$ we have $[W]=\deg(W)h^2$. Moreover, for a codimension one foliation $\F$ on $\mathbb{P}^n$, $n\geq 3$, of degree $d\geq 0$, we have
$$c^2_1(N_{\F})=(d+2)^2h^2.$$ This implies that the hypotheses (2), (3) and (4) of Theorem \ref{Levi} are satisfied for codimension one foliations of $\mathbb{P}^n$. Hence, we can state the following corollary.
\begin{corollary}
Let $\F$ be a codimension one singular holomorphic foliation on $\mathbb{P}^n$, $n\geq 3$, tangent to an irreducible real-analytic Levi-flat hypervariety $M\subset \mathbb{P}^n$. Suppose that $\sing_2(\F)\subset M$. Then there exists an irreducible component $Z$ of $\sing_2(\F)$ such that it contains some dicritical point $p\in\sing(M)$. Moreover, $\F$ has a non-constant meromorphic first integral at $p$.
\end{corollary}


\begin{thebibliography}{99}
\bibitem{abate}
Abate, M., Bracci, F., and Tovena, F.: Index theorems for holomorphic maps and foliations. Indiana University Mathematics Journal 57, no. 7 (2008): 2999-3048. 
\bibitem{aleksandrov}
Aleksandrov, A. G.: Multidimensional residue theory and the logarithmic De Rham Complex. Journal of Singularities, Volume 5, (2012), 1-18.
\bibitem{baum}
Baum, P., and Bott, R.: Singularities of holomorphic foliations. J. Differential Geom. 7 (1972), 279-432.
\bibitem{andres}
 Beltr\'an, A., Fern\'andez-P\'erez, A., and Neciosup, H.: Existence of dicritical singularities of Levi-flat hypersurfaces and holomorphic foliations. Geom. Dedicata (2018) 196: 35. https://doi.org/10.1007/s10711-017-0303-4
 
\bibitem{brunella} Brunella, M.: Singular Levi-flat hypersurfaces and codimension one foliations. Ann. Sc. Norm. Super. Pisa Cl. Sci. (5) vol. VI, no. 4 $(2007)$, 661-672.
\bibitem{index} Brunella, M: Some remarks on indices of holomorphic vector fields. Publicacions Matem\`atiques, vol. 41, no. 2, $(1997)$, 527-544.
\bibitem{perrone}
Brunella, M., and Perrone, C.: Exceptional singularities of codimension one holomorphic foliations. Publ. Mat. 55 (2011), 295-312. 
\bibitem{CS}
Camacho, C., and Sad, P.: Invariant varieties through singularities of
vector fields. Annals of Mathematics, vol. 115, no. 3, $(1982)$, 579-595.
\bibitem{cano}
Cano, F.: Reduction of the singularities of codimension one singular foliations in dimension three. Annals of Mathematics, vol. 160, no 3, $(2004)$, 907-1011.
\bibitem{cartan}
Cartan, E.: Sur la g\'eom\'etrie pseudo-conforme des hypersurfaces de l'espace de deux variables complexes. Ann. Mat. Pura Appl., 11 (1) (1933), 17-90.
\bibitem{alcides}  Cerveau, D., and Lins Neto, A.:
   Local Levi-flat hypersurfaces invariants by a codimension one holomorphic foliation. American Journal of Mathematics, vol. 133 no. 3, $(2011)$, 677-716. doi.org/10.1353/ajm.2011.0018
 \bibitem{machado}
 Corr\^ea, M., and Machado, D.: GSV-index for holomorphic Pfaff systems. Pre-publication 2016. https://arxiv.org/abs/1611.09376v3
 \bibitem{correa}
 Corr\^ea, M., and Louren\c{c}o, F.: Determination of Baum-Bott residues of higher codimensional foliations. Asian Journal of Mathematics vol. 3 no. 3, $(2019)$. doi: http://dx.doi.org/10.4310/AJM.2019.v23.n3.a8

\bibitem{generic} Fern\'andez-P\'erez, A.: On Levi-flat hypersurfaces with generic real singular set. J. Geom. Anal. (2013) 23: 2020. doi.org/10.1007/s12220-012-9317-1
\bibitem{projective}
Fern\'andez-P\'erez, A.: Levi-flat hypersurfaces tangent to projective foliations. J Geom. Anal. (2014) 24: 1959. doi.org/10.1007/s12220-013-9404-y
\bibitem{gmira}
Gmira, B.: Sur les feuilletages holomorphes singuliers de codimension 1. Publicacions Matem\`atiques, vol 36 (1992), 229-240. 
\bibitem{GSV}
G\'omez-Mont, X., Seade, J., and Verjovsky, A.:  The index of a holomorphic flow with an isolated singularity. Math. Ann. (1991) 291: 737. doi.org/10.1007/BF01445237
\bibitem{izawa} 
Izawa, T.: Residues of codimension one singular holomorphic distributions. Bull. Braz. Math. Soc., New Series (2008) 39: 401. doi.org/10.1007/s00574-008-0013-5
\bibitem{jouanolou} J.P. Jouanolou: \'Equations de Pfaff Alg\'ebriques. Lecture Notes in Mathematics, vol. 708. Springer, Berlin (1979) (in French).

\bibitem{variation} Khanedani, B., and Suwa, T.: First variations of holomorphic
forms and some applications. Hokkaido Math. J. 26 (1997), no. 2, 323--335. doi:10.14492/hokmj/1351257968. 
\bibitem{singularlebl} Lebl, J.: Singular set of a Levi-flat hypersurface is Levi-flat. Math. Ann. (2013) 355: 1177. doi.org/10.1007/s00208-012-0821-1
\bibitem{suwa}
Lehmann, D., and Suwa, T.: Generalization of variations and Baum-Bott residues for holomorphic foliations on singular varieties. International Journal of Mathematics 10:03  (1999), 367-384. 
\bibitem{cs_lins}
Lins Neto, A.: Complex codimension one foliations leaving a compact submanifold invariant. Dynamical system and bifurcation theory, Pitman Research Notes in Mathematics Series, vol. 160 (1987), 1-31. 

\bibitem{lins}
Lins Neto, A.: Algebraic solutions of polynomial differential equations and foliations in dimension two, in ``\textit{Holomorphic Dynamics,}'' (Mexico, 1986), Springer, Lectures Notes 1345, (1988). 192-232.

\bibitem{pinchuk}
Pinchuk, S., Shafikov, R., and Sukhov, A.: Dicritical singularities and laminar currents on Levi-flat hypersurfaces. Izv Math, 81 (5) 2017 doi.org/10.1070/IM8582
\bibitem{saito}
Saito, K.: Theory of logarithmic differential forms and logarithmic vector fields. J. Fac. Sci. Univ. Tokyo Sect. IA Math. 27 (1980), no. 2, 265-291.  
\bibitem{sancho}
Sancho De Salas, F.: Residues of a Pfaff system relative to an invariant subscheme. Transactions of the American Mathematical Society, 352(9) (2000), 4019-4035.
\bibitem{shafikov}
 Shafikov, R., and Sukhov, A.: Germs of singular Levi-flat hypersurfaces and holomorphic foliations. Comment. Math. Helv. 90 (2015), 479-502. 
 

\bibitem{suwa_livro}
Suwa, T.: Indices of vector fields and residues of singular holomorphic foliations. Hermann, Editeurs des Sciences et des Arts, (1998).
\bibitem{Junya}
Yoshizaki, J.: On the structure of the singular set of a complex analytic foliation. J. Math. Soc. Japan 50 (1998), no. 4, 837--857. doi:10.2969/jmsj/05040837. 
\end{thebibliography}
\end{document}